\documentclass[11pt]{article} 
\usepackage[english]{babel}

\usepackage{appendix}
\usepackage{xcolor}
\usepackage{geometry}
\usepackage{amsthm,amsmath,amssymb}
\usepackage{algorithm}
\usepackage{algorithmic}
\usepackage{hyperref}
\usepackage{bbm}
\usepackage{mathrsfs}
\usepackage{graphicx}  
\usepackage{epstopdf}
\usepackage{float}
\usepackage{comment}
\usepackage{url}
\usepackage{tikz}
\usepackage{caption}
\usepackage{graphicx} 
\usepackage{float} 
\usepackage{subfigure}
\usepackage{booktabs}
\usepackage{indentfirst}
\usepackage{authblk}

\newtheorem{theorem}{Theorem}[section]
\newtheorem{lemma}{Lemma}[section]

\newtheorem{corollary}{Corollary}[section]

\numberwithin{figure}{section}
\numberwithin{algorithm}{section}

\usepackage{xpatch}
\makeatletter
\AtBeginDocument{\xpatchcmd{\@thm}{\thm@headpunct{.}}{\thm@headpunct{}}{}{}}
\makeatother

\title{\vspace{-0.5in} \bf   Maximizing Throughput in an M/G/1 Queue with Customer Abandonments  \vspace{10pt}}
\author{{\normalsize Runhua Wu and  Hayriye Ayhan} \\ \small H. Milton Stewart School of Industrial and Systems Engineering\\
\small Georgia Institute of Technology\\
\small Atlanta, GA 30332-0205, U.S.A. \vspace{10pt}\\ 
{\normalsize Douglas G. Down} \\
\small Department of Computing and Software\\ \small  McMaster University
\\
\small Hamilton, Ontario, Canada}

\begin{document}
\maketitle

\begin{abstract}
This paper studies the problem of identifying the optimal server assignment policy in single-server queues with customer abandonment. We consider a system with Poisson arrivals and exponentially distributed patience times. We show that when service times follow either an Erlang-$K$ or a hyperexponential distribution and the decision maker can observe the phase of a customer's service time, the Shortest Remaining Expected Processing Time (SREPT) policy maximizes the long-run average throughput, independent of the abandonment rate.
\end{abstract}
\setlength{\parindent}{2em}
\section{Introduction}

Queueing systems with customer abandonments arise in many real-world service environments such as call centers, emergency departments, healthcare systems, and perishable inventory models. In these systems, customers may abandon if their waiting time exceeds their patience level. Despite this wide range of applications, there are relatively few papers that focus on the exact analysis and control of queues with abandonments. Our objective in this paper is to address this issue by studying the problem of server scheduling in queues with abandonments.

Server scheduling to minimize mean response time in a preemptive queue is a classic problem in queueing theory. When there are no customer abandonments and service times are known, Schrage \cite{Schrage} showed that Shortest-Remaining-Processing-Time (SRPT) minimizes the number of customers in the system sample-path-wise among all anticipating policies in queues with general service time distribution and arrival process. Grosof et al.\ \cite{izzy}  showed that SRPT is also asymptotically optimal for multiserver queues under heavy load. However, determining or estimating a customer’s exact remaining processing time is difficult or impossible in many applications, in which case SRPT cannot be implemented.

For non-anticipating service disciplines where service times are unknown but the service time distribution is known, Righter et al.\ \cite{Righter} showed that the FCFS (First Come First Served) service discipline minimizes the mean waiting time in queues with NBUE (New-Better-than-Used-in-Expectation) service time distribution. Gittins \cite{GIP} applied the optimal solution of a multi-armed bandit problem to a Markovian queueing system and obtained the Gittins Index Policy that minimizes the mean waiting time among all non-anticipating preemptive policies in any $M/G/1$ queueing system. Gittins defined a rank function that maps a customer’s age (the amount of service that it has received), to a rank that denotes its priority. However, hidden in its simple framework is a major obstacle: computing the rank function from the service time distribution requires solving a nonconvex optimization problem for every possible age. Although the optimization can be simplified for specific classes of service time distributions (see Aalto et al.\ \cite{Gittens}), it is intractable in general. As a result, it is natural to ask how another non-anticipating scheduling policy, Shortest-Remaining-Expected-Processing-Time (SREPT), performs in queueing systems. Scully et al.\ \cite{Scully} showed that the SREPT policy is near optimal  and  computationally more tractable than the Gittins Index Policy in an $M/G/1$ queue. They also derived a monotonic SREPT policy, which is the only non-Gittins scheduling policy known to have a constant-factor approximation ratio for mean response time. The interested reader is referred to the references in Scully et al.\ \cite{Scully} for other common server scheduling policies that have been considered in $M/G/1$ queues. 

None of the papers discussed above have considered scheduling in the presence of customer abandonments. When abandonments are added to the model, long-run average throughput becomes a nontrivial performance metric (as it can depend on the scheduling policy). One could try to approach the problem using sample-path techniques as in \cite{Righter,Schrage}, who use interchange arguments to show that at every point in time, a particular policy (SRPT in \cite{Schrage}, FCFS in \cite{Righter}) maximizes the average throughput. (The results in those papers are actually stronger, but results for average throughput would suffice for our purposes.) Unfortunately, with customer abandonments such interchange arguments no longer appear to be possible. It is then natural to consider other approaches.

Many recent papers consider approximation methods or asymptotic analysis for server scheduling in queues with multiple customer classes and customer abandonments. Aalto \cite{Aalto1} used the Whittle index approach for a relaxed multi-server scheduling problem with impatient customers and Decreasing-Hazard-Rate (DHR) service times, where customers bring different rewards and holding costs to the system. In the relaxed version of the single-server problem, he shows that the Least Attained Service policy maximizes the long-run average throughput, independent of the abandonment rate. Atar et al.\ \cite{Atar} showed that  $c\mu/\theta$-rule asymptotically maximizes the long-run average gain in many server queues with Poisson arrivals and customer abandonments. Larrañaga et al.\ \cite{Lar} approximated the same stochastic system using a deterministic fluid model and derived a heuristic policy that combines the $c\mu$ and $c\mu/\theta$ rules. Ayesta et al.\ \cite{Ayesta} studied a clearing system and obtained priority rules resembling the $c\mu/\theta$-rule by applying a Markov Decision Process (MDP) formulation and analytically solving cases with one or two customer classes.  Dong and Ibrahim \cite{jing} proved that in systems with multiple servers, a single customer class, and  customer abandonments,  SRPT discipline asymptotically (in the many-server overloaded regime) maximizes the system throughput, among all scheduling disciplines. 

Only a limited number of studies examine exact optimal server scheduling policies in queues with customer abandonments.  Down et al. \cite{Down} considered a two-competing-queues reward system where the two classes have equal service rates. They showed that if one customer class has both a larger abandonment rate and a higher reward, then prioritizing this customer class is optimal.  Bhulai et al.\ \cite{sunjay} generalized these results to $K$ competing queues with customer abandonments and identified conditions guaranteeing the optimality of a strict priority rule. Cao \cite{Cao} considered a similar model with multiple servers and exponential service-time distributions and derived sufficient conditions under which a list policy is optimal under both non-preemptive and preemptive assumptions. Chen et al.\ \cite{chen2} considered a clearing system with exponentially distributed service times and customer abandonments and provided necessary and sufficient conditions for the optimal non-preemptive policy to be an index policy.

We consider a single-stage queueing system to which customers arrive according to a Poisson process with rate $\lambda>0$. At any time, the server may either remain idle or choose any customer in the system to serve. The system has infinite buffer space, and each customer in the system (including the one in service) may abandon the system after an exponentially distributed amount of time with rate $\theta>0$. We consider two types of service-time distributions. In the first case, all customers have an Erlang-$K$ service-time distribution with rate $\mu>0$ in each phase. In the second case, customers have  hyper-exponential service-time distribution with rates $\mu_1,\dots,\mu_K>0$, and probabilities $p_1,\dots,p_K$ where $\sum_{k=1}^Kp_k=1$ and the phase of each customer in the system is observable. For both cases, our objective is to find the server assignment  policy that maximizes the long-run average throughput  among all non-anticipating preemptive scheduling policies.

To the best of our knowledge, this is the first paper that characterizes the exact optimal server assignment policy in a class of $M/G/1$ queues with customer abandonments. More precisely, we show that, independent of the abandonment rate $\theta$, the SREPT policy maximizes the long-run average throughput of queues with Poisson arrivals, Erlang or hyperexponential service time distributions, and exponentially distributed patience time. Though this result may not be surprising, its proof (especially when the service times have Erlang distribution) is difficult. To this end, we model the system as a Continuous-Time Markov Decision Process (CTMDP). Because the underlying problem is not uniformizable, we analyze it directly in continuous time and show that the bias under the SREPT policy satisfies the optimality equation. More precisely, we derive Poisson equations for the first- and second-order differences of the bias function across states in the recurrent class and prove directly—without induction—that all differences are nonnegative. Although standard in functional analysis, this approach is seldom employed in MDP theory, yet it proves especially effective for CTMDPs with infinite state spaces and unbounded transition rates.


\section{Systems With Erlang Service Times} \label{S:Erlang}
We first focus on the case where the service time distribution is Erlang-$K$  with rate $\mu$ in each phase.
\subsection{Problem Formulation}
Consider the stochastic process $\{X^\pi(t):t\geq0\}$ under a policy $\pi\in\Pi$, where $\Pi$ denotes the set of all admissible scheduling policies and $X^\pi(t)=(x_1,x_2,\dots,x_K)\in S=\mathbb{N}^K$ with $x_i$ representing the total number of customers in the system with $i$ remaining phase(s) at time $t$ under policy $\pi$. Let $A(s)$ be the set of possible actions in state $s\in S$. Then for $s=(x_1,x_2,\dots,x_K)$, $A(s)=\{i:x_i>0\}\cup\{0\}$ where action $0$ stands for idling the server. Let $r(s,a)=\mu\mathbbm{1}_{a=1}$ be the reward earned per unit time  and $q_{s_is_j}(a)$ be the transition rate of going from state $s_i$ to state $s_j$ for $s_i\neq s_j$ when action $a$ is chosen in state $s_i$. Then, the following happens in the time interval $[t,t+\mathrm{d}t]$:\par
(i) the system receives an infinitesimal reward $r(s_i,a)\mathrm{d}t$, and \par
(ii) a transition from state $s_i$ to state $s_j$ (with $s_i\neq s_j$) occurs with probability $q_{s_is_j}(a)\mathrm{d}t+o(\mathrm{d}t)$; or the system remains in state $s_i$ with probability $1+q_{s_i}(a)\mathrm{d}t+o(\mathrm{d}t)$,
where $q_{s_i}(a):=-\sum_{s_j\neq s_i} q_{s_is_j}(a)$.\par

The rate matrix of the continuous time Markov chain when action $a\in\mathcal{A}$ is chosen in state $s_i=(x_1,x_2,\dots,x_K)$ is

\[q_{s_is_j}(a)=\left\{
\begin{array}{rcl}
    & \lambda & \mbox{for } s_j=(x_1,x_2,\dots,x_K+1),\\
   & \mu\mathbbm{1}_{a=1}+x_1\theta & \mbox{for } x_1>0,\ s_j=(x_1-1,x_2,\dots,x_K),\\
   & \mu\mathbbm{1}_{a=k} & \mbox{for }k>1,\ x_k>1,\ s_j=(x_1,x_2,\dots,x_{k-1}+1,x_k-1,\dots,x_K),\\
   & x_k\theta & \mbox{for }k>1,\ s_j=(x_1,x_2,\dots,x_k-1,\dots,x_K),\\
  & -\lambda-\mu\mathbbm{1}_{a\neq 0}-\sum_{k=1}^Kx_k\theta & \mbox{for } s_j=s_i,\\
  &0& \mbox{otherwise}.
\end{array}
\right.\]

 The following result from \cite{CTMDP} provides sufficient conditions for the continuous time Markov chain to be regular under any admissible policy.
 \begin{theorem}\label{A1}
For a continuous time MDP, if there exists a function $\omega\geq1$ on $S$ and constants $c_0\neq0$, $b_0\geq0$ and $L_0\geq0$ such that\\
(i) $\sum_{s_j\in S}\omega(s_j)q_{s_is_j}(a)\leq c_0\omega(s_i)+b_0$ for all $(s_i,a)\in S\times A(s_i)$,\\
(ii) $q^*(s):=\sup_{a\in A(s)}\{-q_{ss}(a)\}\leq L_0\omega (s)$ for all $s\in S$.\\
Then the continuous time Markov chain under any admissible policy $\pi\in\Pi$ is regular.
 \end{theorem}
 By setting $\omega(s)=\max\{\sum_{k=1}^K x_k,1\}$ for $s=(x_1,x_2,\dots,x_K)$, $b_0=\lambda+\mu+\theta$, $c_0=-\theta$ and $L_0=\lambda+\theta+\mu$, we can satisfy the two conditions of Theorem \ref{A1}. Note that we have a $K$ dimensional denumerable state space $S$. We define the state $s=(0,0,\dots,0)$ as the first state $s_0$. Then, we order two states $x,y\neq0\in S$ as follows
\[
x \succ y
\;\Longleftrightarrow\;
\begin{cases}
\omega(x) > \omega(y),
\text{ or} \\
\omega(x) = \omega(y)
\ \text{and there exists } m\in\{1,\dots,K\} \text{ such that} \\
 x_1 = y_1,\; x_2 = y_2,\; \dots,\; x_{m-1} = y_{m-1},
\text{ and } x_m < y_m .
\end{cases}
\]
Also, we define
\[x \succeq y\quad \Longleftrightarrow\quad x\succ y \text{ or } x=y.\]
As a result, for any state $s\in S$, starting from $0\to(0,0,\dots,0)\in S$, a unique label (number) is assigned to the state according to the order defined above. Without loss of generality, we use the non-negative integers $\mathbb{N}$ to denote the numbered state space. Intuitively, the SREPT policy is the one that chooses the fastest descent direction (action) in the label of each state.

 Next, for a policy $\pi\in\Pi$, let $\rho_t(B|s)$ be the probability that the policy $\pi$ chooses actions in a Borel set $B\in\mathcal{B}(A(s))$ at time $t$. Then, we define the reward function as
 \[R(t,s,\pi):=\int_{A(s)}r(s,a)\rho_t(\mathrm{d}a|s).\]

 For any policy $\pi\in\Pi$, the long-run average reward when the initial state is $s$, is defined as 
\[g(s,\pi):=\liminf_{T\to\infty}\frac{1}{T}\mathbb{E}_s^\pi[\int_0^T R(t,X^\pi(t),\pi)\mathrm{d}t],\]
\noindent where the expectation is taken with respect to the continuous time Markov chain under policy 
$\pi$. The optimal gain is then defined as
\[g^*(s):=\sup_{\pi\in\Pi}g(s,\pi).\]
Let $\mathbb{B}_\omega$ be the space of all real-valued functions on $S$ with 
\[\sup_{s\in S}\big\{\frac{|g(s,\pi)|}{\omega(s)}\big\}<+\infty.\]
Define $\mathbb{F}$ as the set of all stationary deterministic policies. Then, to guarantee that for all 
$s\in S$, $g^*(s)<\infty$, there exists an optimal policy $\pi^*\in \mathbb{F}$, and a unique solution to the optimality equation that belongs to $\mathbb{B}_\omega$, we need to verify the following conditions of Assumptions 7.1, 7.4, and 7.5 from \cite{CTMDP}. The proof of Theorem \ref{A2} is given in the Appendix.
\begin{theorem}\label{A2}
For the long-run average reward criterion, if\\
(i) $c_0<0$ and $L_0>0$, where $c_0$ and $L_0$ are as defined in Theorem \ref{A1},\\
    (ii) for every $(s,a)\in S\times A(s)$ and some constant $M>0$,
\[|r(s,a)|\leq M\omega(s),\]
(iii) $A(s)$ is compact for all $s\in S$,\\
(iv) for all $s_i,s_j\in S$, the functions $r(s_i,a)$ and $q_{s_is_j}(a)$ are both continuous in $A(s_i)$,\\
(v) there exist a nonnegative function $\omega^\prime$ on $S$ and constants $c^\prime>0$, $b^\prime\geq0$ and $M^\prime>0$ such that 
\[q^*(s)\omega(s)\leq M^\prime \omega^\prime(s)\mbox{ and }\ \]
\[\sum_{s_j\in S}\omega^\prime(s_j)q_{s_is_j}(a)\leq c^\prime \omega^\prime(s_i)+b^\prime\mbox{ for all}\  (s_i,a)\in S\times A(s_i),\]
(vi) for each $\pi\in\mathbb{F}$, the corresponding continuous time Markov chain has a unique invariant probability measure,\\
(vii) for each fixed $\pi\in\mathbb{F}$, for all $s_i,s_k\in S$ such that $s_k\neq s_{i+1}$,
\[\sum_{s_j\succeq s_k}q_{s_is_j}(\pi)\leq \sum_{s_j\succeq s_k}q_{s_{i+1}s_j}(\pi),\]
(viii) for each fixed $\pi\in\mathbb{F}$, and $s_j\succ s_{i_1}\succ (0,0,\dots,0)$, there exist nonzero distinct states $s_{i_2},\dots,s_{i_m}\succeq s_j$ such that
\[q_{s_{i_1}s_{i_2}}(\pi)q_{s_{i_2}s_{i_3}}(\pi)\cdots q_{s_{i_{m-1}}s_{i_m}}(\pi)>0.\]

Then there exists a solution $(g^*,u)\in\mathbb{R}\times \mathbb{B}_\omega(S)$ to the optimality equation
\begin{equation}\label{OE}
    g^*=\sup_{a\in A(s_i)}\{r(s_i,a)+\sum_{s_j\in S}u(s_j)q_{s_is_j}(a)\}\qquad \forall s_i\in S.
\end{equation}
The constant $g^*=g^*(s), \forall s\in S$ and the bias vector $u$ is unique up to additive constants. Moreover, a stationary deterministic policy $\pi\in\mathbb{F}$ is optimal if and only if it attains the maximum in equation (\ref{OE}).
\end{theorem}

Next, we will prove a critical bound for the difference between the values of the bias function evaluated in two different states.
\begin{lemma}\label{L11}
    For any stationary deterministic policy $\pi\in\mathbb{F}$, suppose $x=(x_1,x_2,\dots,x_K),y=(y_1,y_2,\dots,y_K)\in S$, and $x_i\geq y_i$ for all $i=1,\dots,K$. Then we have
    \[0\leq u(x)-u(y)\leq \sum_{i=1}^K (x_i-y_i).\]
\end{lemma}
\begin{proof}
    According to the definition of the true bias function (see equation (7.18) in \cite{CTMDP}), we have
    \[u(x)=\int_0^\infty\big{(}\mathbb{E}_x^{\pi} [r(X^{\pi}(t),\pi)]-g^{\pi}\big{)}\ \mathrm{d}t\]
   where $g^{\pi}=g(s,\pi)$ since the gain is constant under any stationary policy $\pi$. As a result, the difference $u(x)-u(y)$ can be written as
    \[u(x)-u(y)=\int_0^\infty\mathbb{E}_x^{\pi} [r(X^{\pi}(t),\pi)]-\mathbb{E}_y^{\pi} [r(X^{\pi}(t),\pi)]\ \mathrm{d}t.\]
   Consider $X^{\pi}(t)$ in the first expectation and write it as $X^{\pi}(t)=(X^{\pi}(t)-\Delta_{xy})+\Delta_{xy}$ where $\Delta_{xy}=(x_1-y_1,x_2-y_2,\dots,x_K-y_K)$. As a result, all the subsequent transitions from state $x$ can be classified into two different groups: those in $X(t)-\Delta_{xy}$ and those in $\Delta_{xy}$. Note that in the integral, $\Delta_{xy}$ does not depend on the current state. Therefore, the rewards obtained from the process $\Delta_{xy}$ (i.e. $r(\Delta_{xy},\pi)$) is upper bounded by the number of customers in $\Delta_{xy}$. We then have
    \begin{eqnarray}
 u(x)-u(y)&=& \int_0^\infty\mathbb{E}_x^{\pi} [r(X^{\pi}(t)-\Delta_{xy}+\Delta_{xy},\pi)]-\mathbb{E}_y^{\pi} [r(X^{\pi}(t),\pi)]\ \mathrm{d}t\nonumber\\
 &=&\int_0^\infty \mathbb{E}_y^{\pi}[r(X^{\pi}(t),\pi)]+\mathbb{E}^{\pi}[r(\Delta_{xy},\pi)]-\mathbb{E}_y^{\pi} [r(X^{\pi}(t),\pi_t)]\ \mathrm{d}t\nonumber\\
 &=&\int_0^\infty\mathbb{E}^{\pi}[r(\Delta_{xy},\pi)]\leq \sum_{i=1}^K(x_i-y_i).\nonumber   
 \end{eqnarray}
    The inequality follows from the observation that, in the original system, a reward of 1 is received upon each service completion. Furthermore, the rewards $r(\Delta_{xy},\pi)$ are non-negative. Therefore, 
    \[0\leq u(x)-u(y)\leq \sum_{i=1}^K (x_i-y_i).\]
 \end{proof}
 \begin{corollary}\label{BD}
     For any stationary deterministic policy $\pi\in\mathbb{F}$, suppose $x=(x_1,x_2,\dots,x_K),y=(y_1,y_2,\dots,y_K)\in S$. Then 
    \[|u(x)-u(y)|\leq \sum_{i=1}^K |x_i-y_i|.\]
 \end{corollary}
 \begin{proof}
     Consider the third state $z=(x_1\wedge y_1,x_2\wedge y_2,\dots,x_K\wedge y_K)$. We have
     \[|u(x)-u(y)|\leq |u(x)-u(z)|+|u(y)-u(z)|.\]
     From Lemma \ref{L11}, we know that $0\leq u(x)-u(z)\leq\sum_{i=1}^K (x_i-z_i)$ and $0\leq u(y)-u(z)\leq\sum_{i=1}^K (y_i-z_i)$. Then, from the definition of state $z$, we have
   \[|u(x)-u(y)|\leq |u(x)-u(z)|+|u(y)-u(z)|=u(x)-u(z)+u(y)-u(z)\leq \sum_{i=1}^K(x_i-z_i)+\sum_{i=1}^K(y_i-z_i)\leq \sum_{i=1}^K |x_i-y_i|.\]
 \end{proof}
     Next, we prove a lemma that provides the properties of the solution of a system of linear equations, which is crucial for the proofs of the optimality of the SREPT policy.
 
 \begin{lemma}\label{L2}
    Let $B=(b_{ij})_{i,j\in \mathbb{N}}$ be an operator on $\ell^\infty$ such that\\
    (i) for all $i,j\in \mathbb{N}$, $b_{ii}>0$ and $b_{ij}\leq 0$ for $j\neq i$,\\
   (ii) $d_i:=\sum_{j\in \mathbb{N}}b_{ij}>0$, $d:=\inf_{i\in \mathbb{N}} d_i>0$,\\
   (iii) for all $i\in \mathbb{N}$, $\sum_{j> i}|b_{ij}|<\beta_0$ for some constant $\beta_0$.\\
   Then, for any $t\in\ell^\infty$ with $t_i\geq0$ and $u\in\ell^\infty$ satisfying the system of equations
    \[Bu=t,\]
    we have $u\geq0$.
\end{lemma}

\begin{proof} 

Define two operators $C$ and $D$ as:
\[
D = \operatorname{diag}(b_{ii}), \quad C = (c_{ij})_{i,j\in S}, \quad c_{ij} = \begin{cases} -b_{ij}, & i\neq j, \\ 0, & i=j. \end{cases}
\]
Then $C\ge 0$ and $D\geq 0$ entrywise, and $B = D - C$. For any $u\in\ell^\infty$ satisfying $Bu=t$, we have
\[u=D^{-1}\big{(}t+Cu\big{)}.\]

It suffices to show that $\inf_{i\in S} u_i\geq0$. We use contradiction. Since $u\in\ell^\infty$, we know that $m:=\inf_{i\in S}u_i>-\infty$. Suppose $m<0$. First, assume the infimum is attained at some state $i\in S$ (i.e. $u_i=m$), we then immediately have
\[u_i=\frac{1}{b_{ii}}(t_i-\sum_{j\neq i}b_{ij}u_j)\geq \frac{1}{b_{ii}}(t_i-\sum_{j\neq i}b_{ij}m)>m\]
because from assumptions (i) and (ii), we have $-1<\frac{\sum_{j\neq i}b_{ij}}{b_{ii}}<0$, which contradicts $u_i=m$. On the other hand, if the infimum is not attained, then for any $\epsilon>0$, there exists $i\in S$ such that $m<u_i<m+\epsilon$ and $u_j\geq m+\epsilon$ for all $j< i$. Then, we have
\[u_i=\frac{1}{b_{ii}}(t_i-\sum_{j\neq i}b_{ij}u_j)\geq \frac{1}{b_{ii}} \left ( \sum_{j< i}-b_{ij}u_j+\sum_{j> i}-b_{ij}u_j\right )> \left ((m+\epsilon) \sum_{j< i}-\frac{b_{ij}}{b_{ii}}+m\sum_{j> i}-\frac{b_{ij}}{b_{ii}}\right ).\]
From assumptions (i) and (ii) we know that $\sum_{j\neq i}-\frac{b_{ij}}{b_{ii}}\leq 1-\frac{d}{b_{ii}}$. Therefore,
\[u_i>\left ((m+\epsilon) \sum_{j< i}-\frac{b_{ij}}{b_{ii}}+m\sum_{j> i}-\frac{b_{ij}}{b_{ii}}\right )\geq \left ((m+\epsilon) (1-\frac{d}{b_{ii}}+\sum_{j> i}\frac{b_{ij}}{b_{ii}})+m\sum_{j> i}-\frac{b_{ij}}{b_{ii}}\right ).\]
Rearranging the inequality, we have
\[u_i>(m+\epsilon)- \left ((m+\epsilon) (\frac{d}{b_{ii}}-\sum_{j> i}\frac{b_{ij}}{b_{ii}})+m\sum_{j> i}\frac{b_{ij}}{b_{ii}}\right ).\]
Note that $u_i<m+\epsilon$. Then
\[(m+\epsilon)(d-\sum_{j> i}b_{ij})+m\sum_{j> i}b_{ij}>0\]
and 
\[md+\epsilon(d-\sum_{j> i}b_{ij})>0.\]
Finally, from assumption (iii), we have $-\sum_{j> i}b_{ij}$ uniformly bounded above by constant $\beta_0$ for all $i\in S$. Then, letting $\epsilon\to0$ we have $md>0$, which contradicts $m<0$. Therefore, $\inf_{i\in S}u_i\geq0$ and the vector $u$ is non-negative.
\end{proof}
In the remainder of the paper, we adopt a slight abuse of notation: when referring to the entries of a matrix $B=(b_{ij})$ and a vector $t=(t_i)$, we will instead write $B=(b_{s,s'})$ and $t=(t_s)$, when $i$ and $j$ are indices corresponding to states $s$ and $s'$, respectively, according to the ordering defined above.
\subsection{Optimality of the SREPT policy}

To begin with, we introduce the following notation to simplify the state transitions. From now on, we assume state $s=(x_1,x_2,\dots,x_K)$ and define states
\[[s,\lambda]=(x_1,x_2,\dots,x_K+1),\]
\[[s,\theta_i]=(x_1,\dots,x_i-1,\dots,x_K), \text{ for }i=1,2,\dots,K,\]
\[[s,i]=(x_1,\dots,x_{i-1}+1,x_i-1,\dots,x_K),\text{ for }i=2,\dots,K,\]
\[[s,1]=(x_1-1,x_2,\dots,x_K)\text{ and }[s,0]=[s,\theta_0]=s.\]
 Then, the optimality equation in (\ref{OE}) can be written as 
\begin{equation}\label{OE2}
    g^*=\lambda(u([s,\lambda])-u(s))+\theta\sum_{k=1}^K x_k (u([s,\theta_k])-u(s))+\mu\max_{i\in A(s)}\{u([s,i])-u(s)+\mathbbm{1}_{i=1}\}.
\end{equation}

Define $a_s=\min\{i:x_i>0\}$ which is the action that SREPT policy chooses in state $s\succ0\in S$. Note that under the SREPT policy (which in this case is equivalent to FCFS), the set of recurrent states is given by
\[S^*:=\{s\in S|\sum_{k=1}^{K-1}x_k\leq1\}\]
since all arriving customers have $K$ remaining phases to serve. 

Therefore, in order to prove the optimality of the SREPT policy, it is sufficient to show for all $s\succ 0\in S^*$ that $a_s$ attains the maximum in equation (\ref{OE2}) (See Chapter 8.6 in \cite{Put}), i.e.
\[\mathop{argmax}_{i\in A(s)}\{u([s,i])-u(s)+\mathbbm{1}_{i=1}\}=a_s \]
where $u$ is the bias function of the SREPT policy. Then, from the Poisson equation for the SREPT policy (see Proposition 7.11 of \cite{CTMDP}), we have for all $s\succ 0\in S$,
\begin{equation}\label{OES1}
    (\lambda+\theta\sum_{k=1}^Kx_k+\mu)u(s)+g=\lambda u([s,\lambda])+\theta\sum_{k=1}^K x_k u([s,\theta_k])+\mu (u([s,a_s])+\mathbbm{1}_{a_s=1}).
\end{equation}
Define for $i,j \in A(s)$
\[\Delta u(s;i;j)=u([s,i])-u([s,j])+\mathbbm{1}_{i=1}-\mathbbm{1}_{j=1}.\]
Then, to show the optimality of the SREPT policy, one needs to show that
\[\Delta u(s;a_s;j)\geq0 \quad \text{for all }j\in A(s),\ s\in S^*.\]
Note that for any $s\in S^*$, there is at most one customer with less than $K$ remaining phases to serve (and that customer with $a_s$ remaining phase is the one that SREPT policy will continue to serve). Therefore, it suffices to show that for all $s\succ 0\in S^*$,
\[\Delta u(s;a_s;0)\geq0 \quad \text{and} \quad  \Delta u(s;a_s;K)\geq0.\]
In traditional Markov decision process theory, these types of inequalities are typically established via induction. However, such arguments would be difficult (if not infeasible) in our current setting. We instead utilize Lemma \ref{L2} to directly prove the non-negativity of the two aforementioned differences.
\subsubsection{Proof of $\Delta u(s;a_s;0)\geq0$}
Throughout our analysis, for notational simplicity, if a transition yields a state that is not in the state space, we still keep it in our expression (for uniformity) as it cancels with a corresponding term. 
\begin{lemma}
For all $s\succ0\in S$, we have    \begin{eqnarray}\label{OES4}
    (\lambda+\theta\sum_{k=1}^Kx_k+\mu)\Delta u(s;a_s;0)&=&\lambda \Delta u([s,\lambda];a_s;0)+\theta\mathbbm{1}_{a_s=1}\nonumber\\
    &+&\theta\bigg{(}\sum_{k=1}^K x_k \Delta u([s,\theta_k];a_s;0)-\Delta u([s,\theta_{a_s}];a_s;0)\bigg{)}\nonumber\\
    &+&\mu \Delta u([s,a_s];a_{[s,a_s]};0).
\end{eqnarray}
\end{lemma}
\begin{proof}
    We first write the Poisson equation in (\ref{OES1}) for state $[s,a_s]$. If $a_s>1$, then the number of customers in state $[s,a_s]$ is the same as the one in state $s$, since in state $[s,a_s]$, there is one more customer with $a_s-1$ remaining phases and one less customer with $a_s$ remaining phases. Therefore, for $a_s>1$, we have

    \begin{eqnarray}\label{OES2}
    (\lambda+\theta\sum_{k=1}^Kx_k+\mu)u([s,a_s])+g&=&\lambda u([[s,a_s],\lambda])\nonumber\\
    &+&\theta\bigg{(}\sum_{k=1}^K x_k u([[s,a_s],\theta_k])+u([[s,a_s],\theta_{a_s-1}])-u([[s,a_s],\theta_{a_s}])\bigg{)}\nonumber\\
    &+&\mu (u([[s,a_s],a_s-1])+\mathbbm{1}_{a_s-1=1}).
\end{eqnarray}
If $a_s=1$, then there is one less customer in state $[s,a_s]$ than in state $s$. As a result, we have
\begin{eqnarray}\label{OES3p}
    (\lambda+\theta(\sum_{k=1}^Kx_k-1)+\mu)u([s,1])+g&=&\lambda u([[s,1],\lambda])\nonumber\\
    &+&\theta\bigg{(}\sum_{k=1}^K x_k u([[s,1],\theta_k])-u([[s,1],\theta_{1}])\bigg{)}\nonumber\\
    &+&\mu (u([[s,1],a_{[s,1]}])+\mathbbm{1}_{a_{[s,1]}=1}).
\end{eqnarray}
Adding $\theta u([s,1])$ and $(\lambda+\theta\sum_{k=1}^Kx_k+\mu)$ on both sides of equation (\ref{OES3p}) yields
\begin{eqnarray}\label{OES3}
    (\lambda+\theta\sum_{k=1}^Kx_k+\mu)(u([s,1])+1)+g&=&\lambda (u([[s,1],\lambda])+1)\nonumber\\
    &+&\theta\bigg{(}\sum_{k=1}^K x_k (u([[s,1],\theta_k])+1)+u([s,1])-\left(u([[s,1],\theta_{1}])+1\right)+1\bigg{)}\nonumber\\
    &+&\mu (u([[s,1],a_{[s,1]}])+1+\mathbbm{1}_{a_{[s,1]}=1}).
\end{eqnarray}
     Note that in equation (\ref{OES2}) for $a_s>1$, $[[s,a_s],\theta_{a_s-1}]=[s,\theta_{a_s}]$. Therefore, for $x_{a_s}>1$,
     \[u([[s,a_s],\theta_{a_s-1}])-u([[s,a_s],\theta_{a_s}])=-\Delta u([s,\theta_{a_s}];a_s;0).\]
     Similarly, in equation (\ref{OES3}), for $x_1>1$,
     \[u([s,1])-u([[s,1],\theta_{1}]+1)=-\Delta u([s,\theta_1];1;0).\]
    Subtracting (\ref{OES1}) from (\ref{OES2}) (for $a_s>1$) or (\ref{OES3}) (for $a_s=1$) and using the definition of $\Delta u(s;i;j)$, we have
     \begin{eqnarray}
    (\lambda+\theta\sum_{k=1}^Kx_k+\mu)\Delta u(s;a_s;0)&=&\lambda \Delta u([s,\lambda];a_s;0)+\theta\mathbbm{1}_{a_s=1}\nonumber\\
    &+&\theta\bigg{(}\sum_{k=1}^K x_k \Delta u([s,\theta_k];a_s;0)-\Delta u([s,\theta_{a_s}];a_s;0)\bigg{)}\nonumber\\
    &+&\mu \Delta u([s,a_s];a_{[s,a_s]};0).\nonumber
\end{eqnarray}
\end{proof}
\begin{corollary}\label{C1}
    For all $s\succ 0\in S$, we have
    \[\Delta u(s;a_s;0)\geq0.\]
\end{corollary}
\begin{proof}
    In equation (\ref{OES4}), for all $s,s^\prime\in S$, we define 
    \[b_{s,s}=\lambda+\theta\sum_{k=1}^Kx_k+\mu,\ b_{s,[s,\lambda]}=-\lambda,\ b_{s,[s,\theta_k]}=-\theta(x_k-\mathbbm{1}_{k=a_s})\text{ for }k=1,\dots,K,\]
    \[b_{s,[s,a_s]}=-\mu,\ t_s=\theta\mathbbm{1}_{a_s=1}\text{ and }b_{s,s^\prime}=0\text{ otherwise}.\]
    Then, we have
    \begin{eqnarray}\label{OES42}
    b_{s,s}\Delta u(s;a_s;0)+b_{s,[s,\lambda]} \Delta u([s,\lambda];a_{s};0)&&\nonumber\\
    +\sum_{k=1}^K b_{s,[s,\theta_k]} \Delta u([s,\theta_k];a_s;0)+b_{s,[s,a_s]}\Delta u([s,a_s];a_{[s,a_s]};0)&=&t_s.
\end{eqnarray}
Since $a_s=a_{[s,\lambda]}$ and for $k=1,\dots,K$ with $b_{s,[s,\theta_k]}<0$, $a_s=a_{[s,\theta_k]}$, we obtain a system of linear equations for the vector $\left(\Delta u(s;a_s;0)\right)_{s\succ 0\in S}$, satisfying
\[B\Delta u=t,\]
where $B=(b_{s,s^\prime})_{s,s^\prime\succ0\in S}$ and $t=(t_s)_{s\succ0\in S}$ satisfy conditions (i) - (iii) of Lemma \ref{L2}. Note that from Corollary \ref{BD}, we have $\Delta u\in \ell^\infty$, and $t$ is a non-negative vector. Therefore, from Lemma \ref{L2}, we know that $\Delta u(s;a_s;0)\geq0$ for all $s\succ0\in S$.
\end{proof}
Note that Corollary \ref{C1} shows that the server should work on the customer with $a_s$ remaining phases  rather than idling in all states $s\succ0\in S$ (not only for $s\succ0\in S^*$).

\subsubsection{Proof of $\Delta u(s;a_s;K)\geq0$}
 Without loss of generality, in this section we assume that $a_s<K$ and $x_K>0$, because otherwise either $\Delta u(s;K;K)=0$ (by definition) and the result holds trivially or there is only one customer in the system and there is nothing to prove.

In order to simplify the notation, for $s\succ0\in S,\ 0<i<j\leq K,\ i,j\in A(s)$, define 
\[\Delta^2 u(s;i;j):=\Delta u([s,\theta_{i}];0;j)+\Delta u([s,\theta_j];i;0).\]

\begin{lemma}
    For all $s\succ 0\in S^*$, \begin{eqnarray}\label{OES72}
    (\lambda+\theta(x_K+1)+\mu)\Delta u(s;a_s;K)&=&\lambda \Delta u([s,\lambda];a_s;K)+\theta\mathbbm{1}_{a_s=1}+\theta \Delta^2u(s;a_s;K)\nonumber\\
    &+&\theta(x_K-1)\Delta u([s,\theta_K];a_s;K)\nonumber\\
    &+&\mu \Delta u([s,a_s];a_{[s,a_s]};K).
\end{eqnarray}
\end{lemma}
\begin{proof}
Since $a_s<K$, in state $[s,K]$, SREPT policy serves the customer with $a_s$ remaining phases. As a result from (\ref{OES1}) with some algebra, we have
\begin{eqnarray}\label{OES4.5}
    (\lambda+\theta(x_K+1)+\mu)u([s,K])+g&=&\lambda u([[s,K],\lambda])\nonumber\\
    &+&\theta\bigg{(}(x_K-1) u([[s,K],\theta_K])+u([[s,K],\theta_{K-1}])+u([[s,K],\theta_{a_s}])\bigg{)}\nonumber\\
    &+&\mu (u([[s,K],a_s])+\mathbbm{1}_{a_s=1}).
\end{eqnarray}
Subtracting (\ref{OES4.5}) from (\ref{OES2}) (for $a_s>1$) or (\ref{OES3}) (for $a_s=1$), we have
\begin{eqnarray}\label{OES7}
    (\lambda+\theta(x_K+1)+\mu)\Delta u(s;a_s;K)&=&\lambda \Delta u([s,\lambda];a_s;K)\nonumber\\
    &+&\theta (x_K-1)\Delta u([s,\theta_k];a_s;K)\nonumber\\
    &+&\theta\left(u([[s,a_s],\theta_{a_s-1}])+u([[s,a_s],\theta_{K}])-u([[s,K],\theta_{K-1}])-u([[s,K],\theta_{a_s}])\right)\nonumber\\
    &+&\mu \Delta u([s,a_s];a_{[s,a_s]};K).
\end{eqnarray}
Next, we simplify the abandonment terms in the third line of equation (\ref{OES7}). For $a_s>1$, we have
\begin{eqnarray}
    &&u([[s,a_s],\theta_{a_s-1}])+u([[s,a_s],\theta_{K}])-u([[s,K],\theta_{K-1}])-u([[s,K],\theta_{a_s}])\nonumber\\
    =&&u([s,\theta_{a_s}])-u([[s,\theta_{a_s}],K])+u([[s,\theta_K],a_s])-u([s,\theta_K])\nonumber\\
    =&&\Delta([s,\theta_{a_s}];0;K)+\Delta([s,\theta_K];a_s;0)\nonumber\\
    =&&\Delta^2 u(s;a_s;K)\nonumber
\end{eqnarray}
where the first equality holds since $[[s,a_s],\theta_{a_s-1}]=[s,\theta_{a_s}]$, $[[s,K],\theta_{a_s}]=[[s,\theta_{a_s}],K]$ and $[[s,a_s],\theta_K]=[[s,\theta_K],a_s]$.

For $a_s=1$, we have
\begin{eqnarray}
    &&u([s,1])+u([[s,1],\theta_{K}])-u([[s,K],\theta_{K-1}])-u([[s,K],\theta_{1}])\nonumber\\
    =&& u([s,1])-u([[s,1],K])+u([[s,\theta_K],1])-u(s,\theta_K)\nonumber\\
    =&&\Delta u([s,\theta_1];0;K])+\Delta u([s,\theta_K];1;0)+1\nonumber\\
    =&&\Delta^2 u(s;1;K)+1\nonumber
\end{eqnarray}
where the first equality holds since $[[s,K],\theta_{1}]=[[s,1],K]$ and $[[s,1],\theta_K]=[[s,\theta_K],1]$.

As a result, we have,
\begin{eqnarray}(\lambda+\theta(x_K+1)+\mu)\Delta u(s;a_s;K)&=&\lambda \Delta u([s,\lambda];a_s;K)+\theta\mathbbm{1}_{a_s=1}+\theta \Delta^2u(s;a_s;K)\nonumber\\
    &+&\theta(x_K-1)\Delta u([s,\theta_K];a_s;K)\nonumber\\
    &+&\mu \Delta u([s,a_s];a_{[s,a_s]};K).\nonumber
\end{eqnarray}
\end{proof}
We again have a system of linear equations for $\Delta u(s;a_s;K),\ s\succ0\in S^*$ whose coefficients satisfy the conditions in Lemma \ref{L2} (i) - (iii). However, the constant vector in equation (\ref{OES72}) for $s\in S^*$ is $\theta\mathbbm{1}_{a_s=1}+\theta \Delta^2u(s;a_s;K)$, which is not necessarily non-negative. Therefore, we will use a different set of variables to prove the non-negativity of  $\Delta u(s;a_s;K)$.

 For $s\succ0\in S$, $0<i<j\leq K$, $i,j\in A(s)$, we define 
 \begin{equation}\label{D12}
    \delta u(s;i;K):=\Delta u(s;i;K)-q^{i}p_{i,K},
\end{equation}
\begin{equation}\label{D2}
    \delta^2 u(s;i;j):=\Delta ^2 u(s;i;j)+q^{i-1}(\frac{1}{K}-p_{i,j})
\end{equation}
and
\begin{equation}\label{D11}
\delta u(s;0;j):=\Delta u(s;0;j)+\frac{1}{K}+\frac{\theta}{\mu}-p_{0,j},
\end{equation}
where $q=\frac{\mu}{\mu+\theta}$ and $p_{i,j}=\frac{\theta}{\mu}\frac{j-i-1}{K}$. In what follows, our objective is to show that $\delta u(s; a_s; K) \geq 0$, which immediately implies the non-negativity of $\Delta u(s; a_s; K)$. However, as will become apparent in the derivation, establishing $\delta u(s; a_s; K) \geq 0$ requires proving the non-negativity of $\delta^2 u(s; a_s; j)$ for all $0 < j \leq K$, as well as the non-negativity of $\delta u(s; 0; a_s)$. We start with  the following lemma.
 \begin{lemma}\label{L33}
 For all $s\succ0\in S^*$,
     \begin{eqnarray}\label{OES90}(\lambda+\theta(x_K+1)+\mu)\delta u(s;a_s;K)&=&\lambda \delta u([s,\lambda];a_s;K)+\theta \delta^2u(s;a_s;K)+t^\prime_s\nonumber\\
    &+&\theta(x_K-1)\delta u([s,\theta_K];a_s;K)\nonumber\\
    &+&\mu \delta u([s,a_s];a_{[s,a_s]};K)
\end{eqnarray}
where 
for $a_s>1$,
\[t^\prime_s:=-(\mu+2\theta)q^{a_s}p_{a_s,K}-\theta q^{a_s-1}(\frac{1}{K}-p_{a_s,K})+\mu q^{a_s-1}p_{a_s-1,K}\geq0,  \]
and for $a_s=1$,
\[t^\prime_s:=-(\mu+2\theta) qp_{1,K}-\theta (\frac{1}{K}-p_{1,K})+\theta\geq0.\]
 \end{lemma}
\begin{proof}
Writing $\Delta u$ in terms of $\delta u$ (using (\ref{D12})), $\Delta^2 u$ in terms of $\delta^2 u$ (using (\ref{D2})) and plugging into equation (\ref{OES72}) with the definition of $t^\prime_s$, we get equation (\ref{OES90}). Furthermore, for $a_s>1$,
\begin{eqnarray}
    t^\prime_s&:=&-(\mu+2\theta)q^{a_s}p_{a_s,K}-\theta q^{a_s-1}(\frac{1}{K}-p_{a_s,K})+\mu q^{a_s-1}p_{a_s-1,K}\nonumber\\
    &=&\left(\mu q^{a_s-1}p_{a_s-1,K}-(\mu+\theta)q^{a_s}p_{a_s,K}-\theta q^{a_s-1}\frac{1}{K}\right)+\theta q^{a_s-1}p_{a_s-1,K}-\theta q^{a_s}p_{a_s-1,K}\nonumber\\
    &\geq&\mu q^{a_s-1}p_{a_s-1,K}-(\mu+\theta)q^{a_s}p_{a_s,K}-\theta q^{a_s-1}\frac{1}{K}\nonumber\\
    &=&\mu q^{a_s-1}p_{a_s-1,K}-\mu q^{a_s-1}p_{a_s,K}-\mu q^{a_s-1}\frac{\theta}{\mu}\frac{1}{K}\nonumber\\
    &=&\mu q^{a_s-1}\left( p_{a_s-1,K}-p_{a_s,K}-\frac{\theta}{\mu}\frac{1}{K} \right )=0.\nonumber
\end{eqnarray}
Similarly, for $a_s=1$,
\begin{eqnarray}
    t^\prime_s&:=&-(\mu+2\theta) qp_{1,K}-\theta (\frac{1}{K}-p_{1,K})+\theta\nonumber\\
    &\geq& \theta-(\mu+\theta)qp_{1,K}-\theta\frac{1}{K}\nonumber\\
    &=&\frac{\theta}{K}\geq0.\nonumber
\end{eqnarray}
\end{proof}
Equation (\ref{OES90}) has the same coefficients as equation (\ref{OES72}) but a different variable $\delta u(s;a_s;K)$. Next, since $t_s^\prime\geq0$, we will prove that for all $s\succ0\in S^*$, $\delta^2 u(s;a_s;K)\geq0$ which will imply that the constant vector $t^\prime_s+\theta\delta^2 u(s;a_s;K)$ is non-negative.

In order to show that $\delta^2 u(s;a_s;K)$ is non-negative, we first show that $\delta u(s;0;a_s)$ is non-negative, i.e. $\Delta u(s;0;a_s)$ has a lower bound.

\begin{lemma}\label{L34}
For $s\succ0\in S^*$,
\begin{eqnarray}\label{OES6}
    (\lambda+\theta(x_K+1)+\mu)\delta u(s;0;a_s)&=&\lambda \delta u([s,\lambda];0;a_s)+\widetilde{t}_s\nonumber\\
    &+&\theta x_K \delta u([s,\theta_K];0;a_s)\nonumber\\
    &+&\mu \delta u([s,a_s];0;a_{[s,a_s]})
\end{eqnarray}
where for $a_s>1$
\[\widetilde{t}_s:=(\theta+\mu)(\frac{1}{K}+\frac{\theta}{\mu}-p_{0,a_s})-\mu (\frac{1}{K}+\frac{\theta}{\mu}-p_{0,a_s-1})\geq0,\]
and for $a_s=1$,
\[\widetilde{t}_s:=(\theta+\mu)(\frac{1}{K}+\frac{\theta}{\mu})-\mu (\frac{1}{K}+\frac{\theta}{\mu}-p_{0,K})-\theta\geq0.\]
\end{lemma}
\begin{proof}
Changing the sign on both sides of equation (\ref{OES4}), with some simplification, we have for $s\in S^*$,
\begin{eqnarray}\label{OES5}
    (\lambda+\theta(x_K+1)+\mu)\Delta u(s;0;a_s)&=&\lambda \Delta u([s,\lambda];0;a_s)-\theta\mathbbm{1}_{a_s=1}\nonumber\\
    &+&\theta x_K \Delta u([s,\theta_K];0;a_s)\nonumber\\
    &+&\mu \Delta u([s,a_s];0;a_{[s,a_s]}).
\end{eqnarray} 
Writing $\Delta u$ in terms of $\delta u$ (using (\ref{D11})), plugging into equation (\ref{OES6}), and using the definition of $\widetilde{t}_s$, we get equation (\ref{OES5}). Furthermore, for $a_s>1$,
\begin{eqnarray}
    \widetilde{t}_s&:=&(\theta+\mu)(\frac{1}{K}+\frac{\theta}{\mu}-p_{0,a_s})-\mu (\frac{1}{K}+\frac{\theta}{\mu}-p_{0,a_s-1})\nonumber\\
    &=& \mu (\frac{1}{K}+\frac{\theta}{\mu}-p_{0,a_s})-\mu (\frac{1}{K}+\frac{\theta}{\mu}-p_{0,a_s-1})+\theta (\frac{1}{K}+\frac{\theta}{\mu}-p_{0,a_s})\nonumber\\
    &\geq &\mu (p_{0,a_s-1}-p_{0,a_s})+\theta\frac{1}{K}=0.\nonumber
\end{eqnarray}
Similarly for $a_s=1$,
\begin{eqnarray}
    \widetilde{t}_s&:=&(\theta+\mu)(\frac{1}{K}+\frac{\theta}{\mu})-\mu (\frac{1}{K}+\frac{\theta}{\mu}-p_{0,K})-\theta\nonumber\\
    &=&\theta(\frac{1}{K}+\frac{\theta}{\mu})+\mu p_{0,K}-\theta\nonumber\\
    &=&\theta\frac{\theta}{\mu}\geq0.\nonumber
\end{eqnarray}
\end{proof}
Equation (\ref{OES5}) forms a system of linear equations for the vector $\{\delta u(s;0;a_s)\}_{s\succ0\in S^*}$. The following corollary whose proof is given in the Appendix states that $\delta u(s;0;a_s)\geq0$ for all $s\in S^*$.
\begin{corollary}\label{C2}
    For $s\succ0\in S^*$, we have
    \[\delta u(s;0;a_s)\geq0.\]
\end{corollary}
Next, we obtain a system of linear equations for vector $\delta^2u(s;a_s;j)$. From the definition of $\delta^2 u(s;a_s;j)$ and $\Delta^2 u(s;a_s;j)$, we know that for $j\in A(s)$,

\[\delta^2 u(s;a_s;j)=\Delta u([s,\theta_{a_s}];0;j)+\Delta u([s,\theta_j];a_s;0)+q^{i-1}(\frac{1}{K}-p_{i,j}).\]

Note that if $[s,\theta_{a_s}]$ and $[s,\theta_j]\in S^*$, then $s$ may not be necessarily in $S^*$. Thus, we need to consider a larger set of states which we define as  
\[S^\prime:=\{s\in S|\sum_{k=1}^{K-1}x_k\leq 2\}.\]
We next obtain a system of equations for $\{\delta^2(s;a_s;j) \}_{s\succ0\in S^\prime}$.

\begin{lemma}\label{L35}
For $s\succ0\in S^\prime$, $j=a_{[s,\theta_{a_s}]}$, if $a_s>1$,
\begin{eqnarray}\label{OES91}
    (\lambda+\theta (x_K+\mathbbm{1}_{j<K})+\mu)\delta^2 u(s;a_s;j)&=&\lambda\delta^2 u([s,\lambda];a_s;j)\nonumber\\
    &+&\theta(x_K-\mathbbm{1}_{j=K})\delta^2 u([s,\theta_K];a_s;j)\nonumber\\
    &+&\mu\delta^2 u([s,j,a_s];a_s-1;j-1),
\end{eqnarray}
and if $a_s=1$,
\begin{eqnarray}\label{OES92}
    (\lambda+\theta (x_K+\mathbbm{1}_{j<K})+\mu)\delta^2 u(s;a_s;j)&=&\lambda\delta^2 u([s,\lambda];a_s;j)+\widehat{t}_s\nonumber\\
    &+&\theta(x_K-\mathbbm{1}_{j=K})\delta^2 u([s,\theta_K];a_s;j)
\end{eqnarray}
where
\[\widehat{t}_s:=\mu\Delta u([[s,\theta_j],a_s];a_{[[s,\theta_j],a_s]};0)+\mu\delta u([[s,\theta_{a_s}],j];0;j-1)\geq0.\]

\end{lemma}
\begin{proof}
    Note that if $j=0$, then there is only one customer in the system in state $s$, in which case we know that it is optimal to serve that customer. So, assume $j>0$, then for $s\in S^\prime$, we have $[s,\theta_j]\in S^*$. Furthermore, in equation (\ref{OES4}), for a state $s$  in $S^*$,  $x_k=0$ for $k\neq a_s$ and $k\neq K$. Furthermore,  $x_{a_s}=1$, which cancels with the negative term. As a result, we have

\begin{eqnarray}\label{OES61}
    (\lambda+\theta(x_K+\mathbbm{1}_{j<K})+\mu)\Delta u([s,\theta_j];a_s;0)&=&\lambda \Delta u([[s,\theta_j],\lambda];a_s;0)+\theta\mathbbm{1}_{a_s=1}\nonumber\\
    &+&\theta (x_K-\mathbbm{1}_{j=K}) \Delta u([[s,\theta_j],\theta_K];a_s;0)\nonumber\\
    &+&\mu \Delta u([[s,\theta_j],a_s];a_{[[s,\theta_j],a_s]};0).
\end{eqnarray}
Similarly, for the state $[s,\theta_{a_s}]\in S^*$, since $j=a_{[s,\theta_{a_s}]}$, we can again apply equation (\ref{OES4}) and obtain
\begin{eqnarray}\label{OES62}
    (\lambda+\theta(x_K+\mathbbm{1}_{j<K})+\mu)\Delta u([s,\theta_{a_s}];j;0)&=&\lambda \Delta u([[s,\theta_{a_s}],\lambda];j;0)\nonumber\\
    &+&\theta (x_K-\mathbbm{1}_{j=K}) \Delta u([[s,\theta_{a_s}],\theta_K];j;0)\nonumber\\
    &+&\mu \Delta u([[s,\theta_{a_s}],j];j-1;0).
\end{eqnarray}
Note that for $s\in S^\prime$, under the SREPT policy, $j-1=a_{[[s,\theta_{a_s}],j]}$. Subtracting (\ref{OES62}) from (\ref{OES61}) and noting that $\Delta u(s;i;j)=-\Delta u(s;j;i)$, we have
\begin{eqnarray}\label{OES83}
    (\lambda+\theta (x_K+\mathbbm{1}_{j<K})+\mu)\Delta^2 u(s;a_s;j)&=&\lambda\Delta^2 u([s,\lambda];a_s;j)\nonumber+\theta\mathbbm{1}_{a_s=1}\\
    &+&\theta(x_K-\mathbbm{1}_{j=K})\Delta^2 u([s,\theta_K];a_s;j)\nonumber\\
    &+&\mu\Delta u([[s,\theta_j],a_s];a_{[s,\theta_j,a_s]};0)\nonumber\\
    &+&\mu\Delta u([[s,\theta_{a_s}],j];0;j-1).
\end{eqnarray}
Furthermore, for $a_s>1$,
\begin{eqnarray}
    &&\Delta u([[s,\theta_j],a_s];a_{[[s,\theta_j],a_s]};0)+\Delta u([[s,\theta_{a_s}],j];0;j-1)\nonumber\\
    &=&\Delta u([[s,\theta_j],a_s];a_s-1;0)+\Delta u([[s,\theta_{a_s}],j];0;j-1)\nonumber\\
    &=&\Delta u([[[s,j],a_s],\theta_{j-1}];a_s-1;0)+\Delta u([[[s,j],a_s],\theta_{a_s-1}];0;j-1)\nonumber\\
    &=&\Delta^2 u([[s,j],a_s];a_s-1;j-1).\nonumber
\end{eqnarray}
Therefore, we have for $a_s>1$,
\begin{eqnarray}\label{OES84}
    (\lambda+\theta (x_K+\mathbbm{1}_{j<K})+\mu)\Delta^2 u(s;a_s;j)&=&\lambda\Delta^2 u([s,\lambda];a_s;j)\nonumber\\
    &+&\theta(x_K-\mathbbm{1}_{j=K})\Delta^2 u([s,\theta_K];a_s;j)\nonumber\\
    &+&\mu\Delta^2 u([s,j,a_s];a_s-1;j-1).
\end{eqnarray}

Writing $\Delta^2 u$ in terms of $\delta^2 u$ (using (\ref{D2})) and plugging into equation (\ref{OES84}), we get equation (\ref{OES91}).

Similarly, for $a_s=1$, writing $\Delta^2 u$ in terms of $\delta^2 u$ (using (\ref{D2})), plugging into equation (\ref{OES83}) together and using the definition of $\widehat{t}_s$, we get equation (\ref{OES92}).

Note that from Corollaries \ref{C1} and \ref{C2} we have $\widehat{t}_s\geq0$.
\end{proof}

The following corollary whose proof is given in the Appendix states that $\delta^2 u$ is nonnegative.

\begin{corollary}\label{C3}
    For $s\succ0\in S^\prime$, $a_s<j\leq K$, $j\in A(s)$, we have
    \[\delta^2 u(s;a_s;j)\geq0.\]
\end{corollary}

Next corollary whose proof is also given in the Appendix states the same result for $\delta u$.
\begin{corollary}\label{C4}
    For $s\succ0\in S^*$, we have
    \[\delta u(s;a_s;K)\geq0.\]
\end{corollary}
Finally, from the definition of $\delta u(s;a_s;K)$, we have the following corollary.
\begin{corollary}\label{C5}
    For $s\succ0\in S^*$, we have
    \[\Delta u(s;a_s;K)\geq0.\]
\end{corollary}
Combining Corollaries \ref{C1} and \ref{C5}, we get the optimality of the SREPT policy as stated in the following theorem.
\begin{theorem}
    SREPT policy maximizes the long-run average throughput in queues with Poisson arrivals, Erlang service time distribution, and exponentially distributed customer impatience times.
\end{theorem}

\section{Hyperexponential Service Time} \label{S:hyper}

In this section, we consider a system where the service time distribution is hyperexponential with rates $\mu_1,\dots,\mu_K>0$, and probabilities $p_1,\dots,p_K$ where $\sum_{k=1}^Kp_k=1$. The arrival and abandonment processes are identical to those described in Section \ref{S:Erlang}.  Without loss of generality, assume that the rates of the hyperexponential distribution are ordered as $\mu_1>\mu_2>\cdots>\mu_K$. Although the optimality of the SREPT policy can again be established using the technique in Section \ref{S:Erlang}, we instead present a simpler argument based on the main result of Bhulai et al. \cite{sunjay}. The authors study the problem of optimal server scheduling in a $K$-class queue with abandonments, with the objective of minimizing the long-run average cost. Theorem \ref{T41} states their main result. 

\begin{theorem}\label{T41}
    \textit{[Bhulai et al.\cite{sunjay}]} Suppose that the customers can be ordered such that, for $1\leq i\leq j\leq K$, the following three conditions hold
\begin{eqnarray}
    1. & c_i\geq c_j,\nonumber\\
    2. & c_i\mu_i\geq c_j\mu_j,\nonumber\\
    3. & c_i\mu_i/\theta_i\geq c_j\mu_j/\theta_j,\nonumber
\end{eqnarray}
where $\theta_i$ is the abandonment rate and $c_i$ is the holding cost per unit time for customer type $i$, then the Smallest Index Policy is $\alpha$-discount optimal for any $\alpha>0$, and hence also strongly Blackwell optimal.
\end{theorem}

\begin{theorem}
    SREPT policy maximizes the long-run average throughput in queues with Poisson arrivals, hyperexponential service time distribution, and exponentially distributed customer impatience times.
\end{theorem}
\begin{proof}
When the service time distribution is hyperexponential and the rates of the customers in the system are observable, our problem is equivalent to assigning a server to $K$-competing customer classes. Furthermore, in our model, maximizing the long-run average throughput is the same as maximizing the long-run average gain of a system where customers pay a reward $r$ at the time of arrivals and incur a cost $r$ at the time of abandonment. Note that server scheduling policies do not control the arrivals. As a result, finding the optimal policy that maximizes the long-run average throughput in our model is equivalent to finding the optimal policy that minimizes the long-run average abandonment costs in the model where customers only incur a cost $r$ at the time of their abandonments. Let $c_i=r\theta$ for all $i=1,\dots,K$ be the equivalent abandonment cost per unit time. Then, our model is equivalent to the model in Bhulai et al. \cite{sunjay} with $c_i=r\theta$, $\mu_i=\mu$, $\theta_i=\theta$ for all $i=1,\dots,K$, which satisfies the three conditions in Theorem \ref{T41}. Therefore, we know that the Smallest Index Policy, which is  the SREPT policy in our model, maximizes the long-run average throughput in queues with Poisson arrivals, hyperexponential service time distribution, and exponentially distributed customer impatience times.
\end{proof}

\section{Concluding Remarks}
This paper studies server scheduling in a Markovian queueing system with customer abandonments. We show that the SREPT policy maximizes the long-run average throughput of the system when the service times follow Erlang or hyperexponential distributions and the phase is observable. Future research will focus on finding optimal server assignment policies in $M/G/1$ queues with customer abandonments and more general service time distributions. In particular, our objective is to understand whether throughput-optimal policies can be determined independent of the (homogeneous) abandonment rate.
\section{Acknowledgments}
The work of the second author was supported by the National Science Foundation Grant CMMI-2127778. The work of the third author was supported by the Discovery Grant program of the Natural Sciences and Engineering Research Council of Canada.

\section*{Appendix}
\textbf{Proof of Theorem \ref{A2}}

\begin{proof}
    Since we set $c_0=-\theta<0$ and $L_0=\lambda+\theta+\mu>0$, condition (i) is satisfied.
    
    Note that the reward function in our model is either $\mu$ or $0$. As a result, condition (ii) is satisfied with $M=\mu$.
    
    Since $A(s)$ is a discrete set for each $s\in S$, conditions (iii) and (iv) are satisfied trivially.

    Next, for any state $s=(x_1,x_2,\dots,x_K)\in S$, we choose
    \[\omega^\prime(s):=\omega(s)^2=\max\left\{\left(\sum_{i=1}^K x_i\right)^2,1\right\}.\]
    Then, with $M^\prime=\lambda+\theta+\mu$ and $c^\prime=b^\prime=2\lambda >0$  condition (v) is satisfied.

    Finally, due to customer abandonments, the continuous time Markov chain has a unique invariant probability and is stable under any stationary deterministic policy $\pi\in \mathbb{F}$. Therefore, conditions (vi), (vii) and (viii) are satisfied.
\end{proof}

\textbf{Proof of Corollary \ref{C2}}
\begin{proof}
    In equation (\ref{OES6}), for all $s,s^\prime\in S^*$, define $\widetilde{b}_{s,s^\prime}$ as \[\widetilde{b}_{s,s}=\lambda+\theta (x_K+1)+\mu,\ \widetilde{b}_{s,[s,\lambda]}=-\lambda,\ \widetilde{b}_{s,[s,\theta_K]}=-\theta x_K,\]
    \[\widetilde{b}_{s,[s,a_s]}=-\mu\text{ and }\widetilde{b}_{s,s^\prime}=0\text{ otherwise}.\]
    Then, we have
    \begin{eqnarray}
    \widetilde{b}_{s,s}\delta u(s;0;a_s)+\widetilde{b}_{s,[s,\lambda]} \delta u([s,\lambda];0;a_{s})&&\nonumber\\
    +\widetilde{b}_{s,[s,\theta_K]} \delta u([s,\theta_K];0;a_s)+\widetilde{b}_{s,[s,a_s]}\delta u([s,a_s];0;a_{[s,a_s]})&=&\widetilde{t_s}.\nonumber
\end{eqnarray}
Since $a_s=a_{[s,\lambda]}=a_{[s,\theta_K]}$, we obtain a system of linear equations for the vector $\left(\delta u(s;0;a_s)\right)_{s\succ 0\in S^*}$, satisfying
\[\widetilde{B}\delta u=\widetilde{t}\]
where $\widetilde{B}=(\widetilde{b}_{s,s^\prime})_{s,s^\prime\succ0\in S^*}$ and $\widetilde{t}=(\widetilde{t_s})_{s\succ0\in S^*}$ (defined in Lemma \ref{L34}). Note that $\widetilde{B}$ satisfies conditions (i) - (iii) of Lemma \ref{L2}, $\delta u\in \ell^\infty$ since $\Delta u\in\ell^\infty$, and $\widetilde{t}$ is a non-negative vector. Therefore, from Lemma \ref{L2}, we know that $\delta u(s;0;a_s)\geq0$ for all $s\succ0\in S^*$.
\end{proof}

\textbf{Proof of Corollary \ref{C3}}
\begin{proof}
    In equations (\ref{OES91}) and (\ref{OES92}), for all $s,s^\prime\in S^\prime$, define $\widehat{b}_{s,s^\prime}$ as \[\widehat{b}_{s,s}=\lambda+\theta (x_K+\mathbbm{1}_{j<K})+\mu,\ \widehat{b}_{s,[s,\lambda]}=-\lambda,\ \widehat{b}_{s,[s,\theta_K]}=-\theta (x_K-\mathbbm{1}_{j=K}),\]
    \[\widehat{b}_{s,[s,j,a_s]}=-\mu\mathbbm{1}_{a_s>1}\text{ and }\widehat{b}_{s,s^\prime}=0\text{ otherwise}.\]
    Then, we have
    \begin{eqnarray}
    \widehat{b}_{s,s}\delta^2 u(s;a_s;j)+\widehat{b}_{s,[s,\lambda]} \delta^2 u([s,\lambda];a_s;j)&&\nonumber\\
    +\widehat{b}_{s,[s,\theta_K]} \delta^2 u([s,\theta_K];a_s;j)+\widehat{b}_{s,[s,j,a_s]}\delta^2 u([s,j,a_s];a_s-1;j-1)&=&\widehat{t_s}\mathbbm{1}_{a_s=1}.\nonumber
\end{eqnarray}
Since $a_s=a_{[s,\lambda]}=a_{[s,\theta_K]}$, we obtain a system of linear equations for the vector $\left(\delta^2 u(s;a_s;j)\right)_{s\succ 0\in S^\prime}$, satisfying
\[\widehat{B}\delta^2 u=\widehat{\tau}\]
where $\widehat{B}=(\widehat{b}_{s,s^\prime})_{s,s^\prime\succ0\in S^\prime}$ and $\widehat{\tau}=(\widehat{t_s}\mathbbm{1}_{a_s=1})_{s\succ0\in S^\prime}$ (defined in Lemma \ref{L35}). Note that $\widehat{B}$ satisfies conditions (i) - (iii) of Lemma \ref{L2}, $\delta^2 u\in \ell^\infty$ since $\Delta^2 u\in\ell^\infty$, and $\widehat{t}$ is a non-negative vector. Therefore, from Lemma \ref{L2}, we know that $\delta^2 u(s;a_s;j)\geq0$ for all $s\succ0\in S^\prime$.
\end{proof}

\textbf{Proof of Corollary \ref{C4}}
\begin{proof}
    In equation (\ref{OES90}), for all $s,s^\prime\in S^*$, define ${b}_{s,s^\prime}^\prime$ as \[{b}_{s,s}^\prime=\lambda+\theta (x_K+1)+\mu,\ {b}_{s,[s,\lambda]}^\prime=-\lambda,\ {b}_{s,[s,\theta_K]}^\prime=-\theta (x_K-1),\]
    \[{b}_{s,[s,a_s]}^\prime=-\mu\text{ and }{b}_{s,s^\prime}^\prime=0\text{ otherwise}.\]
    Then, we have
    \begin{eqnarray}
    {b}_{s,s}^\prime\delta u(s;a_s;K)+{b}_{s,[s,\lambda]}^\prime \delta u([s,\lambda];a_s;K)&&\nonumber\\
    +{b}_{s,[s,\theta_K]}^\prime \delta u([s,\theta_K];a_s;K)+{b}_{s,[s,a_s]}^\prime\delta u([s,a_s];a_{[s,a_s]};K)&=&t_s^\prime+\theta\delta^2u(s;a_s;K).\nonumber
\end{eqnarray}
Since $a_s=a_{[s,\lambda]}=a_{[s,\theta_K]}$, we obtain a system of linear equations for the vector $\left(\delta u(s;a_s;K)\right)_{s\succ 0\in S^*}$, satisfying
\[{B}^\prime\delta u={\tau}^\prime\]
where ${B}^\prime=({b}_{s,s^\prime}^\prime)_{s,s^\prime\succ0\in S^*}$ and ${\tau}^\prime=(t_s^\prime+\theta\delta^2u(s;a_s;K))_{s\succ0\in S^*}$ (defined in Lemma \ref{L33}). Note that $B^\prime$ satisfies conditions (i) - (iii) of Lemma \ref{L2}, $\delta u\in \ell^\infty$ since $\Delta u\in\ell^\infty$, and $t^\prime$ a non-negative vector. Therefore, from Lemma \ref{L2}, we know that $\delta u(s;a_s;K)\geq0$ for all $s\succ0\in S^*$.
\end{proof}
\end{document}